\documentclass{article}

\usepackage{hyperref}
\usepackage[bibtex-style]{amsrefs}
\usepackage{amsmath, amsthm, amssymb, amsfonts}
\usepackage[utf8]{inputenc}
\usepackage[T1]{fontenc}
\usepackage{bm}
\usepackage{tikz}
\usepackage[capitalize]{cleveref}
\usetikzlibrary{calc}
\usetikzlibrary{decorations}
\usetikzlibrary{decorations.pathreplacing}
\usetikzlibrary{shapes}
\usetikzlibrary{arrows, automata, decorations.pathreplacing, fit, matrix, patterns, positioning}
\usepackage{tikz-qtree}
\usepackage{xifthen} % For if-then control (used in a Tikz macro).
\usepackage{xspace}
\usepackage{bbold}
\usepackage[mathscr]{euscript}
\usepackage{enumitem}
\usepackage{todonotes}
\usepackage{pgf}
\pgfmathsetseed{\number\pdfrandomseed}

\usepackage{lineno}
\newcommand\absdot[2]{
	\node at #1 {\normalsize $\bullet$};
	\node at #1 [below] {$#2$};
}

\newcommand{\plotperm}[1]{
	\foreach \j [count=\i] in {#1} {
		\absdot{(\i,\j)}{};
	};
}

\newcommand{\plotpermbox}[4]{
	\draw [darkgray, thick, line cap=round]
		({#1-0.5}, {#2-0.5}) rectangle ({#3+0.5}, {#4+0.5});
}

\newcommand{\plotpermborder}[1]{
	\plotperm{#1};
	\foreach \i [count=\n] in {#1} {};
	\plotpermbox{1}{1}{\n}{\n};
}

\newcommand{\plotpermgraph}[1]{
	\foreach \j [count=\i] in {#1} {
		\foreach \b [count=\a] in {#1} {
			\ifthenelse{\a<\i \AND \b>\j}{\draw (\a,\b)--(\i,\j);}{}
		};
	};
	\plotperm{#1};
}

\newcommand{\C}{{\mathcal C}}

\renewcommand{\S}{{\mathcal S}}

\newtheorem{theorem}{Theorem}

\newtheorem{definition}[theorem]{Definition}
\newtheorem{lemma}[theorem]{Lemma}

\newtheorem{proposition}[theorem]{Proposition}
\newtheorem{observation}[theorem]{Observation}
\newtheorem{question}[theorem]{Question}

\newcommand{\Av}{\operatorname{Av}}

\newcommand{\Inv}{\operatorname{Inv}}

\newcommand{\we}{\equiv_{\mathrm{WE}}}
\newcommand{\inv}{\preceq}
\newcommand{\ninv}{\npreceq}
\renewcommand{\leq}{\leqslant}

\title{Wilf-collapse in permutation classes having two basis elements of size three.}
\author{Michael Albert \and Jinge Li}

\begin{document}
\maketitle

\begin{abstract}
We consider permutation classes having two basis elements of size three and one further basis element. We completely classify the possible enumeration sequences of such classes and demonstrate that there are far fewer of them than might be expected in principle.
\end{abstract}

\section{Introduction}

Let $\S = \cup \S_n$ be the union of the sets of permutations of $[n]  = \{1, 2, \dots, n\}$. We represent the elements of $\S$ in \emph{one line notation}, i.e., we think of each element $\pi \in \S_n$ as a sequence $\pi_1 \pi_2 \cdots \pi_n$. In this representation there is a natural partial order $\inv$ on $\S$ defined by $\pi \inv \sigma$ if there is a subsequence $\sigma_{i_1} \sigma_{i_2} \sigma_{i_n}$ of $\sigma$ having the same size as $\pi$ and satisfying $\sigma_{i_s} \leq \sigma_{i_t}$ if and only if $\pi_s \leq \pi_t$. In passing we note that the word ``natural'' in the preceding sentence is not poorly chosen -- this order corresponds to the substructure relation on finite models of the theory of two linear orders. We refer to this order as \emph{involvement} and its complement as \emph{avoidance}, i.e., if $\pi \inv \tau$ we say that $\pi$ is involved in $\tau$, while if $\pi \ninv \tau$ we say that $\tau$ avoids $\pi$.

The study of \emph{permutation classes} can be thought of as the study of the downwards-closed (or hereditary) subsets of $\S$ with respect to the order $\inv$. Given such a class $\C$ note that if we take $X$ to be the set of $\inv$-minimal permutations in the complement of $\C$ then 
\[
\C = \{ \pi \in \S \, : \, \mbox{for all $\chi \in X$, $\chi \ninv \pi$} \}.
\]
We say that the permutations in $\C$ \emph{avoid} the permutations in $X$ and write $\C = \Av(X)$. For a more comprehensive introduction to permutation classes we recommend Vatter's survey \cite{Vatter2015Permutation-cla} and references therein.

Given a permutation class $\C$ we are interested in its \emph{enumeration sequence}, $(c_n)_{n \geq 0}$ or its \emph{generating function} $C(t) = \sum_{n=0}^{\infty} c_n t^n$ where $c_n$ is equal to the cardinality of $\S_n \cap \C$. Two permutation classes are said to be \emph{Wilf-equivalent} if they have the same enumeration sequence (or generating function). 

There are a variety of trivial Wilf-equivalences that arise from the automorphisms of $(\S, \inv)$. These automorphisms form a dihedral group of order 8 and can be generated by reversal (of the corresponding words), and inverse. They are most naturally understood geometrically as acting on the set of points $(i, \pi_i)$ within a square $[0, n+1] \times [0, n+1]$ under the normal symmetries of the square. The first instance of a non-trivial Wilf-equivalence is that the pair of classes $\Av(312)$ and $\Av(321)$ are both enumerated by the Catalan numbers. 

The main question we wish to address is:

\begin{question}
Given a class $\C = \Av(X)$, for which $\pi$ and $\tau$ in $\C$ is it the case that $\C \cap \Av(\pi)$ and $\C \cap \Av(\tau)$ are Wilf-equivalent?
\end{question}

For convenience we introduce the notation $\Av_\C(\pi)$ for $\C \cap \Av(\pi)$ and refer to such classes as \emph{principal subclasses} of $\C$. We also use $\we$ to denote the equivalence relation of Wilf-equivalence and, when context is clear we may refer to Wilf-equivalence of $\pi$ and $\tau$ rather than of the principal classes they define.

In \cite{Simion1985Restricted-perm}, one of the seminal papers for the study of permutation patterns, Simion and Schmidt considered specifically permutation classes of the form $\Av(X)$ where all the permutations in $X$ are of size $3$. When $X$ consists of three or more elements, these sets are quite restricted indeed and not terribly interesting. When $X$ is a singleton then as noted above the two (up to trivial Wilf-equivalence) types are both enumerated by the Catalan numbers. Wilf-equivalence between principal subclasses of $\Av(231)$ was considered in \cite{Albert2015A-general-theor}, where a necessary (and conjecturally sufficient) condition for Wilf-equivalence was introduced for these classes. The case of $\Av(321)$ is more complex for various reasons, mostly having to do with the complicated nature of the involvement relation on these permutations. Here we concentrate on Wilf-equivalences among principal subclasses of $\Av(X)$ where $X$ consists of precisely two permutations of size 3. Subclasses of such classes have already been considered in some generality in \cite{Vatter2002Permutations-av} and implicitly elsewhere. There are systematic and automatic methods to determine their individual enumerations, but no general arguments about when their enumerations coincide.

Most results concerning non-trivial Wilf-equivalences have been ad hoc and quite limited, typically dealing with a specific pair of classes, a small finite set of classes, or very simple infinite families. We are interested in following a more systematic approach aimed at discovering whole families of Wilf-equivalences. In particular, we are interested in the phenomenon of \emph{Wilf-collapse}. We say that a class $\C$ exhibits a Wilf-collapse if the sequence, $w_n$ defined as the number of equivalence classes of $\we$ on the permutations in $\C$ of size $n$ has the property that $w_n = o(c_n)$, i.e., there are far fewer enumeration sequences for principal subclasses of $\C$ than there might be in principle. The aforementioned paper \cite{Albert2015A-general-theor} illustrates a Wilf-collapse in $\Av(231)$ demonstrating there that $w_n = o(2.5^n)$ while $c_n^{1/n} \to 4$.

If one of the automorphisms of $\S$ takes a set $X$ to a set $Y$ then the two classes $\Av(X)$ and $\Av(Y)$ are isomorphic as partially ordered sets and hence their behaviour with respect to Wilf-equivalences is the same. Furthermore, the class $\Av(123, 321)$ contains no permutations of size greater than 4 so is not of interest to us in terms of classifying Wilf-equivalences therein. Thus, in our investigation, we may take one of the elements of $X$ to be $312$. Even then there are still some symmetries available leaving only four classes that we need to consider:
\[
\mbox{$\Av(312, 123)$, $\Av(312, 213)$, $\Av(312, 231)$, and $\Av(312, 321)$.}
\]

In the following sections we consider Wilf-equivalences among principal subclasses of these classes one by one. The first two are very straightforward and we will demonstrate a Wilf-collapse directly. In the latter two cases we describe an easily determined equivalence relation on the class which turns out to be the same as that of Wilf-equivalence. To that end we introduce the following definition:

\begin{definition}
An equivalence relation $\sim$ on a class $\C$ is \emph{sound for Wilf-equivalence} if $\pi \sim \tau$ implies $\Av_{\C} (\pi) \we \Av_{\C} (\tau)$. We say that $\sim$ is \emph{sound and complete for Wilf-equivalence} if it is sound and $\Av_{\C} (\pi) \we \Av_{\C} (\tau)$ implies $\pi \sim \tau$.
\end{definition}

Before studying the classes individually we introduce a little more notation which we make use of in several cases.

%\section{Notation and preliminary observations}

\begin{definition}
Let $\pi \in \S_n$ and $\tau \in \S_m$. The \emph{sum}, $\pi \oplus \tau \in \S_{n+m}$ is the permutation defined by:
\[
(\pi \oplus \tau)(k) = 
\left\{
\begin{array}{ll}
\pi(k) & \mbox{for $k \leq n$,} \\
n + \tau(k-n) & \mbox{for $k > n$.}
\end{array}
\right.
\]
A permutation is \emph{$\oplus$-indecomposable} if it is not the sum of two non-empty permutations. 
% An \emph{$\oplus$-indecomposable alphabet} for a permutation class $\C$ is a set of symbols in one to one correspondence with the $\oplus$-indecomposable permutations of $\C$.
\end{definition}

In one line notation, $\pi \oplus \tau$ is formed by concatenating $\pi$ and $n + \tau$ where by $n + \tau$ we mean the sequence formed by adding $n$ to each element of the sequence $\tau$. Since the operation $\oplus$ is easily seen to be associative, any permutation has a unique representation as a sum of $\oplus$-indecomposable permutations and so we can represent permutations in a class as words over the alphabet consisting of the $\oplus$-indecomposable permutations of that class (just by suppressing the operation $\oplus$ itself). From now on we make no distinction between this representation and the corresponding permutations generally using upper case Roman letters to represent words and lower case ones to represent letters.

There is another operation $\ominus$ (which we shall use only briefly) where in one line notation, $\pi \ominus \tau$ is formed by concatenating $\pi + m$ and $\tau$. The two operations are illustrated below:

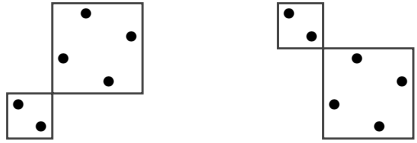
\begin{center}
	\begin{tikzpicture}[scale=0.3]
		\begin{scope}[shift={(0,0)}]
			\plotpermborder{2,1};
		\end{scope}
		\begin{scope}[shift={(2,2)}]
			\plotpermborder{2,4,1,3};
		\end{scope}
		\node at (3,-1) {$21 \oplus 2413 = 214635$};
		\begin{scope}[shift={(12,4)}]
			\plotpermborder{2,1};
		\end{scope}
		\begin{scope}[shift={(14,0)}]
			\plotpermborder{2,4,1,3};
		\end{scope}
		\node at (15,-1) {$21 \ominus 2413 = 652413$};
	\end{tikzpicture}
\end{center}

When working with the plus-indecomposable words representing permutations, we can recognise involvement greedily.

\begin{proposition}
\label{pro:greedy}
Let $\C$ be a permutation class and $W$ and $V$ be two words over the $\oplus$-indecomposable permutations of $\C$. Suppose that $W = wW'$ ($w$ a single letter). Then $W \inv V$ if and only if the following two conditions are satisfied:
\begin{itemize}
\item
$w$ is involved in some letter of $V$, and
\item
if $v_i$  is the first such letter, $V^R$ is the suffix of $V$ following (but not including) $v_i$,  $A$ is the longest prefix of $W$ with $A \inv v_i$, and $W = AB$; then $B \inv V^R$.
\end{itemize}
\end{proposition}

\begin{proof}
If $B \inv V^R$ then it is obvious that $AB \inv V$ simply by embedding $A$ into $v_i$ and the remainder into $V^R$. The converse is almost as immediate. Let an embedding of $W$ into $V$ be given. Either it includes $w$ embedding into $v_i$ or it occurs entirely in $V^R$. In the latter case we have an embedding of $B$ into $V^R$. But in the former case no part of $B$ can embed into $v_i$ either by the maximality of $A$, and so $B \inv V^R$ regardless.
\end{proof}

Of course this proposition also applies to prefixes of $W$ rather than single letters (by inductive application) and also suffixes (by arguing from the right rather than from the left). When the relation $\inv$ on the $\oplus$-indecomposable permutations in $\C$ is sufficiently limited it will be a powerful tool in restricting how permutations of $\C$ can be involved in one another.

%\begin{proposition}
%\label{pro:greediness}
%Let $\C$ be a permutation class and $W$ and $V$ be two words over the $\oplus$-indecomposable permutations of $\C$. Suppose that there are prefixes $A$ of $W$ (say $W = AB$) and $A^{L}$ of $V$ (say $V = A^L B'$) with the properties that
%\begin{itemize}
%\item
%$A \leq A^L$, 
%\item
%$A$ is not involved in any shorter prefix of $V$ than $A^L$, and 
%\item
%no longer prefix of $W$ than $A$ is involved in $A^L$. 
%\end{itemize}
%Then $W \leq V$ if and only if $B \leq B'$.
%\end{proposition}
%
%\begin{proof}
%That $B \leq B'$ implies $W \leq V$ is immediate. So suppose that $W \leq V$. Consider the factorisation given by the preceding proposition and suppose that the subword $W_1 W_2 \cdots W_j$ is the shortest one which contains $A$. Further factor $W_j = X Y$ so that $A = W_1 W_2 \cdots W_{j-1} X$. If $Y$ is not the empty word, then $v_j$ must not occur in the prefix $A^L$ of $V$ or else the second property would be violated. Also by the third property, $v_t$ for $t > j$ cannot occur in the prefix $A^L$. So, in either case, there is an embedding of $B = Y W_{j+1} \cdots W_{k}$ into $B'$.
%\end{proof}
%
%The previous proposition shows that when we seek to recognise embeddings of words of $\oplus$-indecomposable permutations we may, up to a point, do so greedily. It is particularly useful when there are few (or no) embeddings of sums of multiple $\oplus$-indecomposable permutations into single $\oplus$-indecomposable permutations. Of course there is a corresponding version that refers to suffixes.

\section{$\Av(312, 123)$}
\label{sec:312-123}

Throughout this section $\C = \Av(312, 123)$. The enumeration sequence of $\C$ has $c_n = \binom{n}{2} + 1$ and its permutations are characterised as being formed from up to three decreasing segments arranged as shown below.

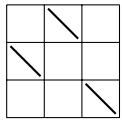
\begin{center}
	\begin{tikzpicture}[scale=0.5]
		\draw (0,0) grid (3,3);
		\draw[thick] (0.1, 1.9) -- (0.9,1.1);
		\draw[thick] (1.1, 2.9) -- (1.9,2.1);
		\draw[thick] (2.1,0.9) -- (2.9,0.1);
	\end{tikzpicture}
\end{center}

In other words, every permutation in $\C$ is of the form $(\alpha \oplus \beta) \ominus \gamma$ where $\alpha$, $\beta$, and $\gamma$ are decreasing permutations. We can therefore describe the permutations in this class by triples $(a,b,c)$ of non-negative integers describing the number of elements belonging to each segment (in left to right order according to the diagram above). In order to avoid ambiguity in the representation of a decreasing sequence of size $c$ (which could be written as either $(x,0,y)$ or $(0,x,y)$ for any $x + y = c$) we also insist that either both or neither of $a$ and $b$ should equal $0$. With these conventions, the order relation in $\C$ is a minor modification of the usual order on $\mathbb{N}^3$ given by:
\begin{align*}
(a, b, c) &\leq (a',b',c') \: \mbox{if $a \leq a'$, $b \leq b'$ and $c \leq c'$; and} \\
(0, 0, c) &\leq (a',b',c') \: \mbox{if $c \leq a' + c'$ or $c \leq b' + c'$}.
\end{align*}

For this class it turns out that for all $n > 1$ there are exactly two Wilf-equivalence classes of permutations of size $n$. 

\begin{theorem}
Let $\pi, \tau \in \C$. Then
\[
\Av_{\C}(\pi) \we \Av_{\C}(\tau)
\]
if and only if $ \left | \pi \right | = \left | \tau \right |$, and either both or neither of $\pi$ and $\tau$ is strictly decreasing.
\end{theorem}

\begin{proof}
If both $\pi$ and $\tau$ are decreasing and of the same length then $\pi = \tau$ and certainly $\Av_{\C}(\pi) \we \Av_{\C}(\tau)$. On the other hand if one (say $\pi$) is decreasing but the other is not then, by the Erd\H{o}s-Szekeres theorem (\cite{Erdos1935A-combinatorial}) $\Av_{\C}(\pi)$ is finite, while $\Av_{\C}(\tau)$ is infinite and so $\Av_{\C}(\pi) \not\we \Av_{\C}(\tau)$. 

So suppose that neither $\pi$ nor $\tau$ is decreasing, but they are of the same size. For $\sigma \in \C$ denote by $\Inv_{\C}(\sigma)$ the set of permutations in $ \C$ involving $\sigma$. Observe that there is a size-preserving bijection between $\Inv_{\C}(\pi)$ and $\Inv_{\C}(\tau)$ if and only if there is a size-preserving bijection between $\Av_{\C}(\pi)$ and $\Av_{\C}(\pi)$. However, if $\pi$ corresponds to the triple $(p_1, p_2, p_3)$ (with $p_1, p_2 \neq 0$ since $\pi$ is not decreasing) and $\tau$ to the triple $(t_1, t_2, t_3)$ then the elements of $\Inv_{\C}(\pi)$ (resp.~$\Inv_{\C}(\tau)$) correspond to the triples $(p_1 + x, p_2 + y, p_3 + z)$ (resp.~$(t_1+x, t_2+y, t_3+z)$) for $(x,y,z) \in \mathbb{N}^3$ and so there is an obvious size-preserving, and in fact order-preserving bijection between them.
\end{proof}

The proof above illustrates an order-preserving bijection between $\Inv_{\C}(\pi)$ and $\Inv_{\C}(\tau)$ for non-decreasing $\pi$ and $\tau$ of equal size. In fact, it is not even required that $\pi$ and $\tau$ have equal size for this bijection to apply. So, in the complementary world where we consider Wilf-equivalences between the complements of ideals this class exhibits an even greater collapse of a sort. It is easy to show that there is in general no order-preserving bijection between $\Av_{\C}(\pi)$ and $\Av_{\C}(\tau)$. For example, take $\pi = 2143$ (corresponding to the triple $(2,2,0)$) and $\tau = 3214$ (corresponding to the triple $(3,1,0)$). The elements of $\Av_{\C}(\pi)$ correspond to triples of the form $(0,0,y)$, $(x,1,y)$ or $(1,x,y)$. Those of $\Av_{\C}(\tau)$ correspond to triples of the form $(0,0,y)$, $(1,x,y)$, or $(2,x,y)$. Every non-decreasing element of $\Av_{\C}(\pi)$ has exactly two covers, while in $\Av_{\C}(\tau)$ the elements $(1,a,b)$ have three covers namely $(2,a,b)$, $(1,a+1,b)$ and $(1,a,b+1)$.

\section{$\Av(312, 213)$}
\label{sec:312-213}

Throughout this section $ \mathcal{C} = \Av(312, 213)$. The permutations in $\C$ can be described by a wedge like structure as shown below.
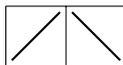
\begin{center}
	\begin{tikzpicture}[scale=0.8]
		\draw (0,0) grid (2,1);
		\draw[thick] (0.1, 0.1) -- (0.9,0.9);
		\draw[thick] (1.1, 0.9) -- (1.9,0.1);
	\end{tikzpicture}
\end{center}

That is, any $\pi \in \mathcal{C}$ can be partitioned into a strictly increasing sequence followed by a strictly decreasing sequence. Alternatively, we can say that the least element of $\C$ must be either the first or the last element (note how the conditions of avoiding $312$ and $213$ imply this -- were the least element not in the first or last position then any element to its left and any element to its right would define one of the two patterns which are supposed to be avoided), and this condition applies recursively on deletion of the least element. It is easy to see that the enumeration of $\C$ is given by $c_n = 2^{n-1}$ (simply by noting that, by the reverse of the previous discussion if we build a permutation in $\C$ from greatest element down then we have, after placing the first element, always two choices for the addition of a new least element). 

We can view $\C$ as being built recursively from the permutation $1$ and closing under the two operations $\pi \mapsto 1 \oplus \pi$ and $\pi \mapsto \pi \ominus 1$.

\begin{observation}
\label{obs:order-in-wedges}
The order relation between elements of $\C$ is described recursively by:
\begin{align*}
1 \oplus \pi \inv 1 \oplus \tau &\iff \pi \inv \tau , \\
1 \oplus \pi \inv \tau \ominus 1 &\iff 1 \oplus \pi \inv \tau , \\
\pi \ominus 1 \inv 1 \oplus \tau &\iff \pi \ominus 1 \inv \tau \mbox{, and} \\
\pi \ominus 1 \inv \tau \ominus 1 &\iff \pi \inv \tau.
\end{align*}
\end{observation}

Our goal is to prove that $\C$ exhibits the greatest possible Wilf-collapse, to wit:

\begin{theorem}
Let $\pi, \tau \in \C$. Then
\[
\Av_{\C}(\pi) \we \Av_{\C}(\tau)
\]
if and only if $ \left | \pi \right | = \left | \tau \right |$.
\end{theorem}

\begin{proof}
We prove this by induction on $|\pi|$. Obviously, the result is trivial for $|\pi| = 1$. So suppose that $|\pi| = |\tau| = n$ and that the result holds for permutations of length less than $n$. We assume that $\pi = 1 \oplus \pi'$ for some $\pi' \in \C$ (the case $\pi = \pi' \ominus 1$ follows an exactly parallel argument). Then $\Av_{\C}(\pi)$ can be described recursively as a disjoint union:
\[
\Av_{\C}(\pi) = \{1\} \cup \left( 1 \oplus \Av_{\C}(\pi') \right) \cup \left( \Av_{\C}(\pi) \ominus 1 \right).
\]

Suppose first that $\tau = 1 \oplus \tau'$, and let $f$ be a size-preserving bijection between $\Av_{\C}(\pi')$ and $\Av_{\C}(\tau')$ which exists by the inductive hypothesis. Using a similar disjoint union decomposition for $\Av_{\C}(\tau)$  we can define a size-preserving bijection, $g$, between $\Av_{\C}(\pi)$ and $\Av_{\C}(\tau)$ recursively as follows:
\begin{align*}
g(1) &= 1,\\
g(1 \oplus \alpha) &= 1 \oplus f(\alpha) \mbox{, and} \\
g(\alpha \ominus 1) &= g(\alpha) \ominus 1.
\end{align*}

If, on the other hand, $\tau = \tau' \ominus 1$ (and $f$ is still a size-preserving bijection between $\Av_{\C}(\pi')$ and $\Av_{\C}(\tau')$) then we can define $g$ by:
\begin{align*}
g(1) &= 1, \\
g(1 \oplus \alpha) &= f(\alpha) \ominus 1 \mbox{, and} \\
g(\alpha \ominus 1) &= 1 \oplus g(\alpha).
\end{align*}
\end{proof}

\section{$\Av(312, 231)$}
\label{sec:312-231}

Throughout this section $ \mathcal{C} = \Av(312, 321)$. The permutations in $\C$ are called \emph{layered permutations} because they can be described as a sequence of decreasing layers as shown below.
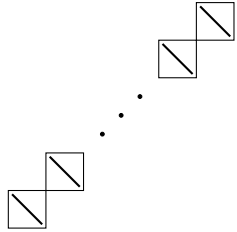
\newcommand{\decbox}[1]{%
\begin{scope}[shift={(#1)}]
   			\draw (0,0) rectangle (1,1);
			\draw[thick] (0.1, 0.9) -- (0.9,0.1);
  		\end{scope}	
}
\begin{center}
	\begin{tikzpicture}[scale=0.5]
		\decbox{0,0}
		\decbox{1,1}
		\decbox{4,4}
		\decbox{5,5}
		\filldraw (2.5,2.5) circle (0.05);
		\filldraw (3,3) circle (0.05);
		\filldraw (3.5,3.5) circle (0.05);
	\end{tikzpicture}
\end{center}

That is, $\C$ is the closure of the class of all decreasing permutations under $\oplus$. To establish this well-known fact, note that the avoidance of $312$ means that all the elements of $\pi$ less than the leftmost element must occur in decreasing order, and the avoidance of $231$ means that all such elements must occur before any element greater than the leftmost element of $\pi$. Therefore $\pi  = \delta \oplus \tau$ for some monotone decreasing permutation $\delta$, and $\tau \in \C$ and then a recursive analysis gives the result above. 

Elements of $\C$ of size $n$ are therefore in one to one correspondence with compositions of $n$ (the sizes of the layers), from which we obtain $c_n = 2^{n-1}$. For the remainder of the section we make no distinction between elements of $\C$ and compositions of $n$ and will represent both as words on the alphabet $\mathbb{N}$. The order relation on compositions $A =  a_1 a_2 \dots, a_k$ and $B = b_1 b_2 \dots, b_m$ that corresponds to involvement in permutations is that $A \leq B$ if there is some subword $b_{i_1} b_{i_2} \dots  b_{i_k}$ of $B$ such that for $1 \leq j \leq k$, $a_j \leq b_{i_j}$. In other words, some subword of $B$ of the same length as $A$ dominates $A$ term by term. This is a somewhat different pattern containment for compositions to that considered in  \cite{Jeli-nek2009Wilf-equivalenc} or \cite{Savage2006Pattern-avoidan} and so the results below do not apply to those contexts.

%In what follows The following observation captures the fact that we can recognise involvement greedily (from the left, but by symmetry also from the right) in $\C$. 
%
%\begin{observation}
%\label{obs:composition-greedy}
%Suppose that $AB, D \in \C$. Then $AB \leq D$ if and only if $A  \leq D$ and, if $A^L$ is the shortest prefix of $D$ with $A \leq A^L$ and $D = A^LE$, then $B \leq E$.
%\end{observation}
%
We now define the equivalence relation $\sim$ on $\C$ to be the finest equivalence relation on $\C$ which satisfies, for all $P, Q \in \C$ (possibly empty), and all positive integers $a$ and $b$:
\begin{align*}
PabQ &\sim PbaQ \mbox{, and} \\
P11Q &\sim P2Q.
\end{align*}

\begin{proposition}
The relation $\sim$ defined above is sound for Wilf-equivalence on $\C$.
\end{proposition}

\begin{proof}
Since both $\sim$ and $\we$ are equivalence relations it suffices to prove that the two defining cases for $\sim$ are sound for Wilf-equivalence, i.e., that
\begin{align*}
\Av(PabQ) &\we \Av(PbaQ) \mbox{, and} \\
\Av(P11Q) &\we \Av(P2Q).
\end{align*}

For the first part note that $\Av(ab) \we \Av(ba)$ due to an underlying symmetry (reversal of the composition, or taking the reverse and complement of the underlying permutation). Let $r$ denote this symmetry. For a composition $S$, if $PQ \inv S$ let $S = P^L M Q^R$ where $P^L$ is the shortest prefix of $S$ with $P \inv P^L$ and $Q^R$ is the shortest suffix of $S$ with $Q \inv Q^R$. This decomposition exists by the extended version of Proposition \ref{pro:greedy} (and its symmetric version on the right). Note also that it is impossible to have $P a \inv P^L$ or $b Q \inv Q^R$ since a single layer cannot embed two or more layers. Now we define a size-preserving bijection $f: \Av(PabQ) \to \Av(PbaQ)$ as follows:
\[
f(S) =
\left\{
\begin{array}{ll}
S & \mbox{if $S \in \Av(PQ)$,} \\
P^L r(M) Q^R &\mbox{if $PQ \inv S$.}
\end{array}
\right.
\]

For the second part observe that $\Av(2)$ consists solely of the compositions all of whose parts are equal to 1 while $\Av(11)$ consists solely of the compositions having only at most one part. So we can construct a size-preserving bijection $g$ here in much the same fashion as above. If $S$ avoids $PQ$ define $g(S) = S$. If $S \in \Av(P2Q)$ but $PQ \inv S$ then $S = P^L 1 1 1 \cdots 1 Q^R$, and we define $g(S) = P^L n Q^R$ where $n$ is the number of 1s appearing in $S$ between $P^L$ and $Q^R$.
\end{proof}

From the conditions on $\sim$ it is clear that each $\sim$ equivalence class contains a partition having no parts of size 2 (the first condition allows us to permute parts arbitrarily, and the second one allows us to replace any 2s by 1s). Since the number of partitions of $n$ is sub-exponential, while $c_n = 2^{n-1}$, this establishes a Wilf-collapse for $\C$, as the number of $\we$ equivalence classes is at most the number of $\sim$ equivalence classes. In fact we can prove that $\sim$ is complete as well as sound for Wilf-equivalence.

In order to prove this we will work directly with the generating functions of the classes concerned. For classes $\Av_{\C}(A)$ let $F_A$ denote the generating function for the enumeration sequence of $\Av_{\C}(A)$.  Suppose that $A = a B$. A composition belongs to $\Av_{\C}(A)$ if (and only if) one of the following holds:
\begin{itemize}
\item
all of its parts are less than $a$, or 
\item
it consists of zero or more parts less than $a$, followed by one part of size $a$ or greater, followed by a (possibly empty) sequence of parts which are in $\Av_{\C}(B)$.
\end{itemize}

Passing to generating functions and using the basic techniques of symbolic combinatorics (\cite{Flajolet2009Analytic-combin}) we get:
\begin{equation}
\label{eqn:gf-recurrence-first-layer}
F_{a B} = \frac{1}{1 - t - t^2 - \cdots - t^{a-1}} \left(  1 + \frac{t^a F_{B}}{1-t} \right)
\end{equation}
Note also that this implies:
\[
t^a F_B = \left( \left( 1 - t - t^2 - \cdots - t^{a-1} \right) F_A  - 1 \right) (1-t) .
\]

From this it follows immediately that:
\begin{lemma}
\label{lem:cancel-first-part}
In $\C$, $a A \we a B$ if and only if $A \we B$.
\end{lemma}

To finish the proof of completeness we need one more observation. For a positive integer $a$, let $r_a$ denote the least positive solution of:
\[
1 - t - t^2 - \cdots - t^{a-1} = 0.
\]

\begin{lemma}
\label{lem:growth-rate}
Let $a > 2$ be a positive integer, and let $B$ be a partition whose greatest part, $b$, is less than $a$. Then $\lim_{t \to r_a^{-}} F_B$ exists and is finite. On the other hand, for any partition $A = a A'$, $\lim_{t \to r_a^{-}} F_A = \infty$.
\end{lemma}

\begin{proof}
We prove the first part by induction on the greatest part of $B$ and the sum of $B$. The base case $B = 1$ is trivial since $F_1 = 1$. So suppose the result holds for all partitions having lesser largest part, or the same largest part and lesser sum than $B$ does. Let $B = bC$. Then Equation \ref{eqn:gf-recurrence-first-layer} (applied to $bC$), the inductive hypothesis, and the fact that
\[
1 - r_a \leq 1 - r_a - r_a^2 - \cdots - r_a^{b-1} = r_a^b + r_a^{b+1} + \cdots + r_a^{a-1} > 0
\]
imply that the result holds for $B$.

For the second part, consider Equation \ref{eqn:gf-recurrence-first-layer} applied to $a A'$. As $t \to r_A^{-}$ the first factor tends to $\infty$ while the second factor is positive and greater than 1. So the product tends to infinity.
\end{proof}

Now we obtain:

\begin{theorem}
\label{thm:completeness-c4}
The relation $\sim$ is sound and complete for Wilf-equivalence in $\C$.
\end{theorem}

\begin{proof}
As soundness has already been established we need prove only completeness. The partitions having no parts of size 2 form a set of equivalence class representatives for $\sim$, so it suffices to prove that if $A$ and $B$ are partitions of the same positive integer having no parts of size 2, then $F_A = F_B$ implies $A = B$. We establish this by induction, noting that the case where the greater of the two greatest parts of $A$ and $B$ is 1 is trivial.

So suppose that $A$ and $B$ are partitions having no parts of size 2, that $F_A = F_B$, that at least one of $A$ and $B$ has a part of size greater than 2, and that the result holds for all partitions of lesser integers. Without loss of generality suppose that the greatest part, $a$, of $A$ is at least as great as that of $B$. Since $F_A = F_B$, by Lemma \ref{lem:growth-rate}, $a$ must also be the greatest part of $B$. But then by Lemma \ref{lem:cancel-first-part} and induction $A = B$ as required.
\end{proof}

\section{$\Av(312, 321)$}
\label{sec:312-321}

Throughout this section $ \mathcal{C} = \Av(312, 321)$. The permutations in $\C$ can be described as sums of individual components of the form $\iota \ominus 1$ with $\iota$ an increasing permutation (possibly empty) as shown below.
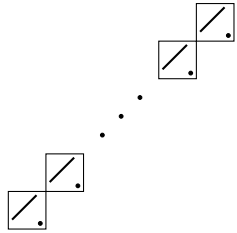
\newcommand{\cycbox}[1]{%
\begin{scope}[shift={(#1)}]
   			\draw (0,0) rectangle (1,1);
			\draw[thick] (0.1, 0.25) -- (0.75,0.9);
			\filldraw (0.85,0.15) circle (0.05);
  		\end{scope}	
}
\begin{center}
	\begin{tikzpicture}[scale=0.5]
		\cycbox{0,0}
		\cycbox{1,1}
		\cycbox{4,4}
		\cycbox{5,5}
		\filldraw (2.5,2.5) circle (0.05);
		\filldraw (3,3) circle (0.05);
		\filldraw (3.5,3.5) circle (0.05);
	\end{tikzpicture}
\end{center}

Once again there is an obvious correspondence between the elements of $\C$ of size $n$ and compositions of $n$. However, with respect to avoidance, parts of size 1 behave very differently from parts of other sizes since a part of size $k+1$ (i.e.~corresponding to a summand $23\cdots(k+1)1$) can involve up to $k$ successive parts of size 1. So, in representing the permutations of $\C$ we distinguish between these two types of parts.

Any permutation in $\mathcal{C} = \Av(312, 321)$ is the sum of strictly increasing sequences $a^i = 1 2 \ldots i$ ($i \geq 1$) and sequences of the form $b_j = 2 3 \ldots j 1$ ($j \geq 2$). For example, the permutation $\pi = 213467859$ can be represented as $21 \oplus 12 \oplus 2341 \oplus 1 = b_2 \oplus a^2 \oplus b_4 \oplus a^1$. So there is a one to one correspondence between the permutations in $\C$ and words over the alphabet consisting of symbols $a^i$ and $b_j$ not containing consecutive $a$'s. As in the previous section we henceforth make no distinction between the elements of $\C$ and such words. Further we use upper case letters to refer to such words. A slightly modified form of Proposition \ref{pro:greedy} applies to $\C$ with respect to this representation (we cannot be sure that a single letter of one word will be involved in a single letter of another since for instance $a^3 \inv b_2 b_3$ without being involved in either one -- however it is still the case that in testing for involvement of one word in another we can do so greedily from left to right or right to left).

We now define the equivalence relation $\sim$ on $\C$ to be the finest equivalence relation on $\C$ which satisfies
five rules -- four of which can be thought of as local conditions, and one as a global one:
\begin{align*}
b_i b_j & \sim b_j b_i  \\
a^i b_j & \sim  b_j a^i  \\
a^i b_j a^k & \sim a^k b_j a^i \\
b_2 b_k & \sim a^1 b_k a^1 \\
A  \sim B & \Rightarrow P A Q \sim P B Q \\
\end{align*}
assuming in the final case that $P$ and $Q$ are arbitrary words subject only to the condition that $PAQ$ and $PBQ$ contain no consecutive $a$'s.

\begin{proposition}
The relation $\sim$ defined above is sound for Wilf-equivalence on $\C$.
\end{proposition}

\begin{proof}
The first three conditions are sound for Wilf-equivalence because they represent correspondences introduced by reversing the word representation of a permutation in $\C$. With respect to permutations this corresponds to the automorphism of $\S$ that fixes $231$, namely a reflection through a diagonal running from SE to NW. So the corresponding classes are trivially Wilf-equivalent.

Although it is possible to provide a bijective proof of the soundness of the fourth condition it seems a little simpler to work directly with generating functions in this case. For a word $X$ let $F_X$ denote the generating function of $\Av_{\C}(X)$. Then we have:
\begin{align*}
F_{b_2 b_k} &= 1 + t F_{b_2 b_k} + \frac{t^2}{1-t} F_{b_k} \\
\Rightarrow (1-t) F_{b_2 b_k} &= 1 + t F_{b_2 b_k} + \frac{t^2}{1-t} F_{b_k} \\
\Rightarrow F_{b_2 b_k} &= \frac{1}{1-t} + \frac{t^2}{(1-t)^2} F_{b_k}. 
\end{align*}
The first equation above arises from recognising that the permutations in $\Av(b_2 b_k)$ consist of: the empty permutation, any permutation whose first summand is 1 and the remainder of which avoids $b_2 b_k$, and any permutation whose first summand contains a decrease (there is a unique such summand, $b_j$, for every $j \geq 2$)  and whose remaining summands avoid $b_k$. The following two lines are simple algebraic rearrangements.

On the other hand permutations in $\Av(a^1 b_k a^1)$ are either empty, consist of a single summand, or consist of two or more summands in which case the summands except for the first and last must avoid $b_k$. Collecting the first two possibilities into one and noting there is a unique possible summand of every size greater than or equal to 1 we get:
\[
F_{a^1 b_k a^1} = \frac{1}{1-t} + \frac{t^2}{(1-t)^2} F_{b_k}.
\]
Hence, $F_{b_2 b_k} = F_{a^1 b_k a^1}$ as claimed.

The final condition is sound due to being able to recognise involvement greedily. Namely, to construct a size-preserving bijection $g : \Av_{\C}(PAQ) \to \Av_{\C}(PBQ)$ given such a bijection $f : \Av_{\C}(A) \to \Av_{\C}(B)$ we simply fix all permutation in $\Av_{\C}(PQ)$, while writing words $W$ involving $PQ$ in the form $W = P^L X Q^R$ where $P^L$ is the shortest prefix of $W$ involving $P$ and $Q^R$ the shortest suffix of $W$ involving $Q$ and then setting $g(W) = P^L f(X) Q^R$. The condition that neither $PAQ$ nor $PBQ$ contains consecutive $a$s implies that no letters can share an involvement of parts of $P$ and $A$, nor of parts of $A$ and $Q$ (and likewise with $B$) and so for $P^L X Q^R$ to avoid $PAQ$ it is both necessary and sufficient that $X$ avoid $A$.
\end{proof}

The upshot of all these rules is that any two words $P$ and $Q$ whose underlying multisets of letters are the same are definitely $\sim$ equivalent, and we may also ``trade'' occurrences of $b_2$ for a pair of $a^1$s or vice versa provided that we do not introduce consecutive $a$s. We could nominate as the representative of a $\sim$ equivalence class a pair of partitions, the first representing the indices held on the $b$s, and the second the one held on the $a$s subject to the condition that the length of the second partition is not more than one greater than the length of the first (so that the $a$s can be inserted without creating consecutive $a$s) and the second partition contains at most one $1$. This makes it clear that $\C$ again exhibits a Wilf-collapse as the generating function for pairs of partitions dominates that for the representatives described above and still exhibits sub-exponential growth.

Before embarking on the proof of completeness for $\sim$ a little more preparatory work is required. For $n
 \geq 0$ define the polynomial $p_n(t)$ to be the generating function for the permutations in $\C$ having a longest increasing subsequence of length exactly $n$. These generating functions are polynomials because every permutation in $\C$ avoids $321$ and hence, if it is of length greater than $2n$ contains an increasing subsequence of length at least $n+1$.

\begin{proposition}
The polynomials $p_n(t)$ satisfy:
\begin{align*}
p_0(t) &= 1 \\
p_1(t) &= t + t^2 \\
p_n(t) &= (2t + t^2) p_{n-1}(t) - t^2 p_{n-2}(t) \: \mbox{, for $n \geq 2$.}
\end{align*}
\end{proposition}

\begin{proof}
The first two equalities are immediate. For the third note that a permutation in $\C$ whose longest increasing subsequence has length $n$ can be obtained from one whose longest increasing subsequence has length $n-1$ by: adding a final summand equal to $1$, lengthening a final summand (other than 1) by 1, or adding a final summand equal to $21$. The number of permutations that can be lengthened in the second case is $p_{n-1}(t) - tp_{n-2}(t)$ since they correspond to those permutations enumerated by $p_n(t)$ whose final summand is not 1 (enumerated by $t p_{n-1}(t)$). Thus:
\[
p_n(t) = t p_{n-1}(t) + t (p_{n-1}(t) - t p_{n-2}(t)) + t^2 p_{n-1}(t) = (2t + t^2) p_{n-1}(t) - t^2 p_{n-2}(t)
\]
as claimed.
\end{proof}

Noting that the minimum length of a permutation enumerated by $p_n(t)$ is equal to $n$ we define the polynomials $q_n(t)$ such that $p_n(t) = t^n q_n(t)$. Then the recurrence for these polynomials is:
\[
q_n(t) = (2 + t) q_{n-1}(t) - q_{n-2}(t).
\]

It is easy to verify that the $q_n$ are directly related to the Chebyshev polynomials $U_{2n}$ according to the formula:
\[
q_n(t) = U_{2n} ( (-t)^{1/2}/2).
\]
However, to make use of them in the completeness proof we need only one basic fact concerning the behaviour of their roots.

\begin{proposition}
\label{pro:increasing-roots}
For $n \geq 1$ let $r_n$ denote the greatest real root of $q_n(t)$. Then the sequence $r_n$ is strictly increasing, bounded above by 0, and for all $n \geq 2$, $r_n > -1/2$.
\end{proposition}

\begin{proof}
As $q_n(t)$ has non-negative coefficients and constant term 1, all real roots must be negative. Furthermore $-1 = r_1 < r_2 = \sqrt{5} - 2$, so $r_2 > -1/2$. Now supposing that the result holds inductively for all $k < n$ note that
\[
q_n(r_{n-1}) = (2 + r_{n-1}) q_{n-1}(r_{n-1}) - q_{n-2}(r_{n-1}) < 0
\]
since $q_{n-1}(r_{n-1}) = 0$ and $q_{n-2}(r_{n-1}) > 0$ as $r_{n-1} > r_{n-2}$. Therefore $q_n(t)$ has a root strictly between $r_{n-1}$ and $0$ and hence $r_n > r_{n-1}$.
\end{proof}

We will also need a cancellation lemma. For convenience in the remainder of this section define $I_A(t)$ to be the generating function for permutations in $\C$ that \emph{involve} (rather than avoid) $A$.

\begin{lemma}
\label{lem:cancel}
In $\C$, $a^i A \we a^i B$ if and only if $A \we B$ (assuming neither $A$ nor $B$ begins with $a$) and $b_j A \we b_j B$ if and only if $A \we B$.
\end{lemma}

\begin{proof}
Consider the general structure of a permutation in $\C$ that involves $a^i A$. Its shortest prefix that involves $a^i$ can be obtained from a permutation enumerated by $p_i(t)$ by replacing the final summand by any summand of equal or greater length. Such prefixes are enumerated by $p_i(t)/(1-t)$. So the generating function for permutations involving $a^i A$ is equal to $p_i(t) I_A(t)/(1-t)$. But, the same argument applies to $a^i B$. So we see $I_{a^i A}(t) = I_{a^i B}(t)$ if and only if $I_{A}(t) = I_{B}(t)$ which is equivalent to the first half of the stated result. The same style of argument applies for the other half as well, specifically:
\begin{equation}
\label{eqn:rec-for-bj}
I_{b_j A}(t) = \frac{1}{1-t-t^2-\cdots-t^{j-1}} \cdot \left( \frac{t^j}{1-t} \right) \cdot I_{A}(t).
\end{equation}
Since we can compute $I_{b_j A}(t)$ just from $I_A(t)$ and $j$ (and vice versa) the second half of the result follows.
\end{proof}

\begin{lemma}
If $B \in \C$ contains no letter $a^n$ for any $n \geq 2$, then, for all $n \geq 2$, $I_B(t)$ converges at $r_n$ and is non-zero there.
Suppose that $A \in \C$ and that $a^n$ occurs as a letter in $A$ for some $n \geq 2$, but $a^m$ does not for any $m > n$. Then $I_A(t)$ converges at $r_n$, $I_A(r_n) = 0$ and $I_A(r_m) \neq 0$ for any $m > n$.
\end{lemma}

\begin{proof}
For the first part of the lemma note that the generating function $I_B(t)$ is dominated termwise by the generating function of $\C$ which is $1 + t/(1-2t)$ and therefore the radius of convergence of $I_B(t)$ is at least $1/2$. Since $-1/2 < r_n < 0$ for all $n \geq 2$, $I_B(t)$ converges at $r_n$. Moreover, considering equation \ref{eqn:rec-for-bj} and applying induction it follows that $I_B(r_n) \neq 0$.

Likewise the radius of convergence of $I_A(t)$ is at least $1/2$ so $I_A$ converges at $r_n$. Moreover, using $\sim$ we may assume that $A = a^n B$. Then:
\[
I_A(r_n) = p_n(r_n) I_B(r_n)/(1-r_n) = 0.
\]

Suppose that there were some $A$ as indicated and some $m > n$ with $I_A(r_m) = 0$. We may choose such an example with the minimum possible value of $n$ and of minimum size for that value. Again, using $\sim$, we may assume this counterexample has the form $A = a^n B$. But then
\[
I_A(r_m) = p_n(r_m) I_B(r_m)/(1-r_m) \neq 0
\]
(since $p_n(r_m) \neq 0$ by Proposition \ref{pro:increasing-roots}, and $I_B(r_m) \neq 0$ by either the first half of the lemma or the assumption that we have taken a minimal counterexample). Thus we have a contradiction.
\end{proof}

Now we are finally in a position to prove the completeness of $\sim$.

\begin{theorem}
The equivalence relation $\sim$ is sound and complete for Wilf-equivalence in $\C$.
\end{theorem}

\begin{proof}
Soundness has already been demonstrated so it remains only to prove completeness. To that end suppose that $A \we B$. Suppose that $A$ contains some letter $a^i$ for some $i \geq 2$, and let $n$ be the greatest integer such that $a^n$ occurs in $A$. Without loss of generality (else just interchange $A$ and $B$) we may assume that no letter $a^m$ for $m > n$ occurs in $B$. But if $a^n$ did not occur in $B$ then by the previous lemma $I_A(r_n) = 0 \neq I_B(r_n)$ which would be a contradiction. So, after rearrangement using $\sim$ if need be we can assume that $A = a^n A'$ and $B = b^n B'$. Then by Lemma \ref{lem:cancel}, $A' \we B'$. Thus, if a counterexample to the theorem exists, then there must be $A$ and $B$ containing no $a^n$ for any $n \geq 2$ with $A \we B$ and $A \not \sim B$. However, for such $A$ and $B$ the proof of Theorem \ref{thm:completeness-c4} can be transcribed essentially verbatim as, for words using only the letters $a^1$ and $b_j$ for $j \geq 2$, any involvement must be witnessed letter for letter as it is in the layered case.
\end{proof}

\section{Comments and acknowledgements}

The class $\C = \Av(312, 123)$ considered in Section \ref{sec:312-123} is a class of polynomial growth. The general structure of such classes is relatively well understood (see for instance \cite{Albert2007Permutation-cla,Homberger2016On-the-effectiv,Huczynska2006Grid-classes-an,Kaiser2002On-growth-rates}) and it seems likely that a complete analysis of Wilf-equivalences and Wilf-collapses in polynomial classes is possible. This is the subject of ongoing work of the second author as part of his PhD. 

The soundness results of Sections \ref{sec:312-213} and \ref{sec:312-231} appear in a different form in a paper of Jel\'{i}nek, Mansour and Shattuck (\cite{Jeli-nek2013On-multiple-pat}) but we have included them here both for completeness and because the style of proof is very much different. As far as we are aware the completeness result for the second of these cases is new as are all the results of Section \ref{sec:312-321}.

It seems notable that in Sections \ref{sec:312-231} and \ref{sec:312-321} relatively simple combinatorial arguments establish the soundness of an equivalence relation $\sim$ which is sufficient to demonstrate Wilf-collapse. The first author, V\'{i}t Jel\'{i}nek and Michal Opler have observed further instances of this phenomenon in more general $\oplus$-closed classes which are also the subject of ongoing investigations. However, to obtain the completeness result relies on appealing to analytic techniques which appears to be difficult to generalise.  A similar situation arose in a slightly different context in \cite{Albert2015Equipopularity-}. Furthermore, for the principal sublcasses of $\Av(231)$ considered in \cite{Albert2015A-general-theor} the completeness result is absent (though conjectured on the basis of fairly strong experimental evidence) precisely because we could not find a way to carry out those analytic techniques in that context.

Though there is no remaining trace of it in the exposition, the results of this paper were first suggested experimentally by extensive machine computation.

\bibliography{refs}{}
\bibliographystyle{plain}
\end{document}